\documentclass[11pt]{amsart}
\usepackage[colorlinks=blue,linkcolor=blue]{hyperref}
\usepackage{enumerate}
\usepackage{lineno}
\usepackage{graphicx}
\usepackage{geometry}
\usepackage{enumitem}
\makeatletter
\@namedef{subjclassname@}{%
\textup{}}
\makeatother
\newtheorem{thm}{Theorem}[section]
\newtheorem{thmx}{Theorem}[section]
\newtheorem{cor}[thmx]{Corollary}
\newtheorem{lem}[thm]{Lemma}

\newtheorem{case}{Case}
\newtheorem{subcase}{Subcase}

\renewcommand{\thethmx}{\Alph{thmx}}
\newtheorem{rem}{Remark}
\newtheorem{Que}{Question}
\newtheorem{exa}{Example}

\numberwithin{equation}{section}

\begin{document}

\title[]{On Meromorphic Solutions to a Difference Equation of Tumura-Clunie Type}

\author{JianRen Long*, Xuxu Xiang \quad }

\address{Jianren Long \newline School of Mathematical Sciences, Guizhou Normal University,
Guiyang, 550025, P.R. China. }
\email{longjianren2004@163.com}

\address{Xuxu Xiang \newline School of Mathematical Sciences, Guizhou Normal University,
Guiyang, 550025, P.R. China. }
\email{1245410002@qq.com}

%\address{Shiwen Wang \newline School of Mathematical Sciences, Guizhou Normal University,
%Guiyang, 550025, P.R. China. }
%\email{1759257600@qq.com}

%\address{Zhigao  Qin\newline School of Mathematical Sciences, Guizhou Normal University,
%Guiyang, 550025, P.R. China. }
%\email{qzhigao@qq.com}

\date{}

%%............................................................................................................................................................
%%...........................................................................
%??(abstract)...................................................................
%%............................................................................................................................................................

\begin{abstract}The meromorphic solutions $f$ with $\rho_2(f)<1$ of the non-linear  difference equation
\begin{align*}
f^n(z)+P_d(z,f)=p_1e^{{\lambda_1}z}+p_2e^{{\lambda_2}z}+p_3e^{{\lambda_3}z},
\end{align*}
are characterized in terms of exponential functions  using Nevanlinna theory, under certain conditions on $\lambda_j$ for $j=1,2,3$. Here, $n>2$, $P_d(z,f)$ is a difference polynomial in $f$ of degree $\le n-1$, and $\lambda_j,~p_j\not=0$ for $~j=1,2,3$. These results improve upon those previously obtained by Chen et al.[Bull. Korean Math. Soc. 61, 745-762 (2024)]. Some examples are provided to illustrate these results. Additionally, if $P_d(z,f)$ is a differential-difference polynomial, then under the supplementary condition  $N(r,f)=S(r,f)$, by applying the same proof method, these conclusions still hold.
\end{abstract}

%\thanks{*Corresponding author}

%%.............................................................................................................................................................
%%.........................................................................
%?????(keywords)....................................................................
%%.............................................................................................................................................................

\keywords{Nevanlinna theory; Tumura-Clunie type difference equations; Meromorphic functions; Existence; Order of growth
*Corresponding author. \\
}
%\qquad *Correspording author}
\maketitle
%%..............................................................................................................................................................
%%.....................................................................................
%????(Introduction}.......................................................
%%.................
\section{Introduction and main results }

Let $f$ be a meromorphic function in the complex plane $\mathbb{C}$. Assume that the reader is familiar with the standard notation and basic results of Nevanlinna theory, such as the proximity function $m(r,f),$ the counting function $N(r,f)$ and the characteristic function$~T(r,f)$; for the more details, see \cite{hayman,ccy2003}. A meromorphic
function $g$ is said to be a small function of $f$ if $T(r,g)=S(r,f)$, where $S(r,f)$  denotes any quantity
that satisfies $S(r,f)= o(T(r, f))$ as $r$ tends to infinity, outside of a possible exceptional set of finite logarithmic measure. We use $\rho(f)=\underset{r\rightarrow \infty}{\lim\sup}\frac{\log^+T(r,f)}{\log r}$ and $\mu(f)=\underset{r\rightarrow \infty}{\lim\inf}\frac{\log^+T(r,f)}{\log r}$ to denote the order and lower order of $f$, respectively. The hyper-order of $f$ is defined by  $\rho_2(f)=\underset{r\rightarrow \infty}{\lim\sup}\frac{\log^+\log^+T(r,f)}{\log r}.$   In general, the differential-difference polynomial of $f$ is defined by
\[P_d(z,f)=\sum_{i=1}^{m}b_i(z)[f^{(\nu_{i_1})}(z+c_{i_1})]^{k_{i_1}}\cdots[f^{(\nu_{i_m})}(z+c_{i_m})]^{k_{i_m}} ,\]
where~$k_{ij}$~are nonnegative integers, $d=\max\limits_{1\leq i\leq m}\{k_{i_1}+k_{i_2}+\cdots +k_{i_m}\}$ is the degree of $P_d(z, f)$, coefficients~\(b_i(z)\)~are meromorphic functions.~If~\(c_{i_1}=\dots =c_{i_m}\)=0,~ then~\(P_d(z,f)\)~is the differential polynomial of $f$.~If~$\nu_{i_j}=0$ for$~j=1,~2,...,m$, then~$P_d(z,f)$~is the difference polynomial of $f$.

The study of Tumura-Clunie type complex functional equations originated from a generalization of the theorem of Tumura-Clunie\cite{clu,Tu} given by Hayman\cite[p. 69]{hayman} in the last century.
\begin{renewcommand}{\thethm}{\Alph{thm}}
\begin{thm}[\cite{hayman}]
\label{thA}
 If nonconstant meromorphic functions $f$ and $g$ in the nonlinear differential equation
 \begin{align}
 \label{tc}
 f^n(z)+P_d(z,f)=g(z)
 \end{align}
satisfies $N(r,f)+N(r,\frac{1}{g})=S(r,f)$, where $P_d(z,f)$ is a differential polynomial in $f$ with degree $d\leq{n-1}$,  then we have $g(z)=(f(z)+\gamma(z))^n$, where
$\gamma$ is a small meromorphic function of $f$.
\end{thm}
\end{renewcommand}

Motivated by the fact that $f=\sin z$ satisfies the equation
\begin{align}
\label{c1}
4f^3(z)+3f^{''}(z)=-\sin3 z
\end{align}
 and  Theorem \ref{thA}, scholars have tried to study the nonlinear differential equation \eqref{tc} for $g(z)=p_1e^{a_1z}+p_2e^{a_2z}$, where $p_j,~a_j,~j=1,2$, are nonzero constants, such as \cite{llw2014,lp2011,ccy2006,ccy2004,zyy2021} and so on. In particular, Yang and Li\cite{ccy2004} showed that the equation \eqref{c1} has exactly three non-constant entire solutions: $f_1(z)=\sin z$, $f_2(z)=\frac{\sqrt 3}{2}\cos z-\frac{1}{2}{\sin z}$, and
$f_3(z)=-\frac{\sqrt 3}{2}\rm{cosz}-\frac{1}{2}{\sin z}$.

With the establishment and development of the difference Nevanlinna theory, see \cite{chiang,halburd2006,halburd20062} for the details, the research of difference equations has become a hot topic, see\cite{halburd2007,halburd2017,lk,jrl} and references therein.
 In 2010,  Yang and Laine\cite{ccy2010}  presented the following  theorem on
difference analogues of equation \eqref{c1}.

\begin{renewcommand}{\thethm}{\Alph{thm}}
\begin{thm}[\cite{ccy2010}]
\label{thb}
A nonlinear difference equation
\begin{align*}
f^3(z)+q(z)f(z+1)=c\sin(bz),
\end{align*}
where $q(z)$ is a nonconstant polynomial and $b,~c$ are nonzero constants, does not admit entire solution of finite order. If $q(z)=q$ is a nonzero constant, then the above equation possesses three distinct entire solutions of finite order, provided that $b=3n\pi$ and $q^3=(-1)^{n+1}\frac{27}{4}c^2$ for a nonzero integer n.
\end{thm}
\end{renewcommand}

Motivated by the Theorem \ref{thb} and
 originally in the theorem of Tumura-Clunie\cite{clu,Tu}, which considers $P_d(z,f)$ as a  polynomial in $f$ with constant coefficients and $g=p(z)e^{a(z)}$, where $a(z)$ is an entire function and $p(z)$ is a small function relative to $f$, scholars have  made
attempts to  find meromorphic solutions of equation \eqref{tc} provided $P_d(z,f)$ is  a difference monomial or polynomial in $f$ and $g$ extends to be a linear combination of two or more linearly independent exponential terms. Examples of such studies include \cite{lxm,mzq2022,qdz,wzt2012,xxx,zm,ccy2010}.

In 2021, Liu and Mao\cite{lhf2021} obtained the following result, which can be viewed as the difference analogues of the results in \cite{lp2011,ccy2006}.

\begin{renewcommand}{\thethm}{\Alph{thm} }
\begin{thm}[\cite{lhf2021}]
Let $n\ge2$ be an integer, $P_d(z,f)\not\equiv 0$ be a difference polynomial in $f$ of degree $d\le n-1$ with polynomial coefficients, and let $p_j,~a_j,$ for $~j=1,2$, be nonzero constants satisfying  $\frac{a_1}{a_2}\in \{ -1,\frac{t}{n},\frac{n}{t}:1\le t \le n-1\}$. If difference equation
\begin{align}
\label{D}
f^n(z)+P_d(z,f)=p_1e^{a_1z}+p_2e^{a_2z}
\end{align}
admits a finite order meromorphic solution $f$, then one of the following statements holds.
\begin{itemize}
\item[(1)]$f(z)=\gamma_0(z)+\gamma_1(z)e^{\frac{a_1z}{n}}+\gamma_2(z)e^{\frac{a_2z}{n}}$, and $\frac{a_1}{a_2}=-1$, where $\gamma_i^n=p_i,~i=1,2$, $\gamma_0$ is a polynomial;
    \item [(2)]$f(z)=\gamma_0(z)+\gamma_1(z)e^{{\beta z}}$ and $\frac{a_1}{a_2}=\frac{n}{t}$ (or $\frac{t}{n}$), where $n\beta=a_1$ (or $a_2$), $\gamma_1^n=p_1$ (or $\gamma_2^n=p_2$), $\gamma_0$ is a polynomial. Moreover, if $P_d(z,0)\not\equiv 0$, then $\gamma_0\not\equiv 0$.
\end{itemize}
\end{thm}
\end{renewcommand}

For the case of $P_d(z,f)\equiv 0$ in \eqref{D}, Liu and Mao obtained the following result.

\begin{renewcommand}{\thethm}{\Alph{thm} }
\begin{thm}[\cite{lhf2021}]
\label{thD}
Let $n\ge2$ be an integer,$~a_j$ be nonzero distinct constants,$~p_j(z)$ be nonzero meromorphic functions, $~j=1,2$. Then the following equation
\begin{align}
\label{E}
f^n(z)=p_1e^{a_1z}+p_2e^{a_2z}
\end{align}
can not admits a finite order meromorphic solution $f$ such that $T(r,p_j)=S(r,f),~j=1,2$.
\end{thm}
\end{renewcommand}

Recently, Chen et al.\cite{cmf2024} considered the right
sides of \eqref{D} and \eqref{E} with three linearly independent exponential terms and obtained the following result.

\begin{renewcommand}{\thethm}{\Alph{thm} }
\begin{thm}[\cite{cmf2024}]
\label{thE}
Let $n\ge3$ be an integer, $P_d(z,f)$ be a difference polynomial in $f$ of degree $d\le n-2$ with small functions of $f$ as its coefficients, and let $p_j,~a_j,~j=1,2,3$, be nonzero constants satisfying
$\frac{a_2}{a_1}\in \{ \frac{s}{n}:1\le s \le n-1\}$,  $\frac{a_3}{a_1}\in \{ \frac{t}{n}:1\le t \le n-1\}$, $s\not=t$.
If difference equation
\begin{align}
\label{F}
f^n(z)+P_d(z,f)=p_1e^{a_1z}+p_2e^{a_2z}+p_3e^{a_3z}
\end{align}
admits a finite order meromorphic solution $f$, then $f(z)=\gamma_1+\gamma_2e^{\frac{a_1z}{n}}$, where $\gamma_1$  is a constant, and $\gamma_2^n=p_1$. Moreover, if $P_d(z,0)\not\equiv 0$, then $\gamma_1\not\equiv 0$.
\end{thm}
\end{renewcommand}

If $n=2,~P_d(z,f)\equiv 0$ in \eqref{F}, then the following result was shown.
\begin{renewcommand}{\thethm}{\Alph{thm} }
\begin{thm}[\cite{cmf2024}]
\label{thF}
Let  $~a_j,~j=1,2,3$, be nonzero distinct constants,$~p_j(z),~j=1,2,3$, be nonzero meromorphic functions. If the following equation
\begin{align}
\label{G}
f^2(z)=p_1e^{a_1z}+p_2e^{a_2z}+p_3e^{a_3z}
\end{align}
 admits a finite order meromorphic solution $f$ such that $T(r,p_j)=S(r,f),~j=1,2,3$, then $f(z)=\gamma_ie^{\frac{a_iz}{2}}+\gamma_je^{\frac{a_jz}{2}}$, where $\gamma_i^2=p_i$, $\gamma_j^2=p_j$, $2\gamma_i\gamma_j=p_k$, $a_i+a_j=2a_k$, $\{i,j,k\}=\{1,2,3\}$.
\end{thm}
\end{renewcommand}

For Theorem \ref{thE}, we find that there exists an equation
$f^4(z)+f^2(z+i\pi)+f(z+2i\pi)=2e^{z}+16e^{4z}+4e^{2z}
$  which has a solution $f(z)=2e^z$. Here, $n=4$, $d=2$, and $a_1=1,~a_2=4,~a_3=2$  do not satisfy the conditions $\frac{1}{n}\le\frac{a_2}{a_1}\le\frac{n-1}{n}$ and $\frac{1}{n}\le\frac{a_3}{a_1}\le\frac{n-1}{n}$. Therefore, it is natural to ask the following question.

\begin{Que} What can be said about the solutions of the equation \eqref{F} when the restricted conditions on $a_j$ are weakened, for $ j=1,2,3$?
\end{Que}

By observing the conditions on $n,~d$ in Theorem \ref{thE} and  Theorem \ref{thF}, we can ask the following two questions.

\begin{Que}
If $d=n-1$, does  Theorem \ref{thE} still holds ?
\end{Que}

\begin{Que}
 If the integer $n>2$, does the equation  $f^n(z)=p_1e^{a_1z}+p_2e^{a_2z}+p_3e^{a_3z}$ has meromorphic solutions  $f$ such that $T(r,p_j)=S(r,f)$ for $j=1,2,3$ ?
\end{Que}

In this paper, we mainly consider the above three questions. First of all,  we obtain the following result to answer Question 1.

\setcounter{thmx}{0}
\renewcommand{\thethmx}{\arabic{section}.\arabic{thmx}}
\begin{thmx}
Let $n\ge3$ be an integer, $P_d(z,f)\not\equiv 0$ be a difference polynomial in $f$ of degree $d\le n-2$ with small functions of $f$ as its coefficients, and let $p_j,~a_j,~j=1,2,3$, be non-zero constants such that $0<\frac{a_2}{a_1}=t\le\frac{n-2}{n}$.
If difference equation \eqref{F}
admits a meromorphic solution $f$ with $\rho_2(f)<1$,  then one of the following conclusions holds.
\begin{itemize}
    \item [(1)]$f(z)=\gamma_1e^{\frac{a_1z}{n}}$, $~\frac{a_1}{a_2}=\frac{n}{d_1},~\frac{a_1}{a_3}=\frac{n}{d_2},~\gamma_1^n=p_1,~P_d(z,f)=p_2e^{a_2z}+p_3e^{a_3z}$, or\\    $f(z)=\gamma_3e^{\frac{a_3z}{n}}$,$~\frac{a_3}{a_2}=\frac{n}{d_3},~\frac{a_3}{a_1}=\frac{n}{d_4},~\gamma_3^n=p_3,~P_d(z,f)=p_2e^{a_2z}+p_1e^{a_1z},$ where $d_i\in\{1,\dots,d\}$ are integers, $ i=1,2,3,4$.
    \item [(2)]$f(z)=c_1e^{\frac{a_1z}{n}}+c_0$, $\frac{a_3}{a_1}=\frac{n-1}{n}$, where $c_0\not=0$ is a constant, $c_1^n=p_1,$ $-c_0^n$ equals the constant term of $P_d(z,c_1e^{\frac{a_1z}{n}}+c_0)$, or\\
 $f(z)=c_2e^{\frac{a_3z}{n}}+c_3$,  $\frac{a_1}{a_3}=\frac{n-1}{n}$, where $c_3\not=0$ is a constant, $c_2^n=p_3,$ $-c_3^n$ equals the constant term of $P_d(z,c_2e^{\frac{a_3z}{n}}+c_3)$.
    \item [(3)] $f(z)=t_1e^{\frac{a_1z}{n}}+t_3e^{\frac{a_3z}{n}}$, $a_1+a_3=0$, where $t_1^n=p_1,~t_3^n=p_3$.
\end{itemize}
\end{thmx}

\begin{rem}
If $P_d(z,f)$ is a differential-difference polynomial, then under the supplementary condition  $N(r,f)=S(r,f)$, by applying the proof method of Theorem 1.1, the conclusion of Theorem 1.1 still holds.
\end{rem}

\begin{rem} Theorem E requires  $\frac{a_2}{a_1}\in[\frac{1}{n},\frac{n-1}{n}]$,  $\frac{a_3}{a_1}\in[\frac{1}{n},\frac{n-1}{n}]$.  However, our Theorem 1.1 only needs  $ \frac{a_2}{a_1}\in(0,\frac{n-2}{n}]$. Therefore, the restricted condition on $a_j$ in Theorem 1.1 weaker than  Theorem \ref{thE}, $j=1,2,3$. In addition, Theorem 1.1  relaxes the restricted condition $\rho(f)<\infty$ in Theorem \ref{thE} into $\rho_2(f)<1$.
\end{rem}

For Question 2, we have the following result.

\begin{thmx}
Let $n\ge3$ be an integer, $P_d(z,f)\not\equiv 0$ be a difference polynomial in $f$ of degree $d= n-1$ with small entire function of $f$ as its coefficients, and let $p_j,~a_j,~j=1,2,3,$ be nonzero constants satisfying
$\frac{a_1}{a_2}\in \{ \frac{s}{n}:1\le s \le n-1\}$,  $\frac{a_3}{a_2}\in \{ \frac{t}{n}:1\le t \le n-1\}$, $s\not=t$.
If difference equation \eqref{F}
%\begin{align}
%\label{th1.2}
%f^n(z)+P_d(z,f)=p_1e^{a_1z}+p_2e^{a_2z}+p_3e^{a_3z}
%\end{align}
admits a meromorphic solution $f$ with $\rho_2(f)<1$, then $f(z)=\gamma_1+\gamma_2e^{\frac{a_2z}{n}}$, where $\gamma_2^n=p_2$,  $\gamma_1$ is a small entire function of $f$.
\end{thmx}

\begin{rem}
Theorem 1.2 shows that Theorem \ref{thE} still holds when $P_d(z,f)\not\equiv 0$ is a difference polynomial in $f$ of degree $d= n-1$ with small entire functions of $f$ as its coefficients.
\end{rem}

\begin{rem}
If $P_d(z,f)$ is a differential-difference polynomial, then under the supplementary condition  $N(r,f)=S(r,f)$, by applying the proof method of Theorem 1.2, the conclusion of Theorem 1.2 still holds.

\end{rem}

Finally, we give a negative answer to Question 3.

\begin{thmx}
    Let $n\ge3$ be an integer,$~a_j$ be nonzero distinct constants,$~p_j(z)$ be nonzero meromorphic functions, $j=1,2,3$. Then the following  equation
\begin{align}
\label{th1.3}
f^n(z)=p_1e^{a_1z}+p_2e^{a_2z}+p_3e^{a_3z}
\end{align}
  does not admit a  meromorphic solution $f$ such that $T(r,p_j)=S(r,f),~j=1,2,3$.
\end{thmx}

We can prove the following corollary by using a similar method to that used in the proof of Theorem 1.3,  where $a_3=0$ in equation \eqref{th1.3}.

\begin{cor}
 Let $n\ge3$ be an integer,$~a_i,~i=1,2,$ be nonzero distinct constants,$~p_j(z)$ be nonzero meromorphic functions, $j=1,2,3$. Then   equation
\begin{align*}
f^n(z)=p_1e^{a_1z}+p_2e^{a_2z}+p_3
\end{align*}
  does not admit a  meromorphic solution $f$ such that $T(r,p_j)=S(r,f),~j=1,2,3$.
\end{cor}

\section{Examples}
In this section, we present some examples to demonstrate our results. Firstly, the following three examples are provided to illustrate the conclusions of Theorem 1.1, specifically $(1),~(2),~(3)$ respectively. Additionally, the case where $\frac{a_2}{a_1}=\frac{n-2}{n}$ can occur is also included.
\begin{exa}
$f=2e^z$ satisfies the equation $f^4(z)+f^2(z+i\pi)+f(z+2i\pi)=16e^{4z}+2e^{z}+4e^{2z}
$. Here $n=4$, $d=2$ and $a_1=4,~a_2=1,~a_3=2$. This shows the conclusion $(1)$ of Theorem 1.1.
\end{exa}

\begin{exa}
$f=e^z+1$ satisfies the equation $f^3(z)+f(z+2i\pi)-2=e^{3z}+4e^{z}+3e^{2z}
$. Here $n=3$, $d=1$ and $a_1=3,~a_2=1,~a_3=2$. This shows that the conclusion $(2)$ of Theorem 1.1, and $\frac{a_2}{a_1}=\frac{n-2}{n}$  can occur.
\end{exa}

\begin{exa}
$f=e^{z}+e^{-z}$ satisfies the equation $f^3(z)+\frac{3}{2i}f(z+\frac{i\pi}{2})-\frac{3}{2}f=e^{3z}+3e^{z}+e^{-3z}
$. Here $n=3$, $d=1$ and $a_1=3,~a_2=1,~a_3=-3$. This shows the conclusion $(3)$ of Theorem 1.1.
\end{exa}

Next, we give the following two examples to illustrate the conclusion of Theorem 1.2.

\begin{exa}
Equation $f^3(z)-f^2(z+i\pi)=2e^{2z}+e^{3z}+5e^{z}
$ has a solution $f=e^{z}+1$. Here $n=3$, $d=2$ and $a_1=2,~a_2=3,~a_3=1$. %This shows the conclusion   of Theorem 1.2.
\end{exa}

\begin{exa}
Equation $f^3(z)-3zf^2(z)+f^2(z)+3z^2f(z)-2zf(z)+f(z+2i\pi)-z^3+z^2-z-2i\pi=e^{2z}+e^{3z}+e^{z}
$ has a solution $f=e^{z}+z$. Here $n=3$, $d=2$ and $a_1=2,~a_2=3,~a_3=1$. %This shows the conclusion  of Theorem 1.2.
\end{exa}

Finally, Example 6 is provided to demonstrate the necessary condition regarding $a_1, a_2, a_3 ,$ in Theorem 1.2.
\begin{exa}
\cite[Example 1.4]{cmf2024}
The equation $f^3(z)-f^2(z)+2f(z)-f(z-\ln2)+1=e^{3z}+\frac{3}{2}e^z+e^{-3z}
$ has a solution $f=e^{z}+e^{-z}+1$ that does not satisfy the conclusion of Theorem 1.2. Obviously, $a_1=3,~a_2=1,~a_3=-3$ does not satisfy the conditions $\frac{1}{n}\le\frac{a_1}{a_2}\le\frac{n-1}{n}$ and $\frac{1}{n}\le\frac{a_3}{a_2}\le\frac{n-1}{n}$. This shows the necessary condition on $a_1,a_2,a_3,$ in Theorem 1.2.
\end{exa}

\section{Preliminary lemmas}

n this section, we will collect some preliminary results for proving our results. The first lemma can be found in \cite{lp2011}.
\begin{lem}\cite[Lemma 6] {lp2011}
\label{lm2.2}
Suppose that $f$ is a transcendental meromorphic function, $a, b, c, d,$ are the small functions of $f$  such that $acd\not\equiv0$. If
\begin{align*}
    af^2+bff'+c(f')^2=d,
\end{align*}
then $c(b^2-4ac)\frac{d'}{d}+b(b^2-4ac)-c(b^2-4ac)'+(b^2-4ac)c'=0$.
\end{lem}
\begin{rem}
From \cite[Remark 2.9 ] {cmf2024}, we know that the condition $acd\not\equiv0$ in Lemma \ref{lm2.2} can
be replaced by $cd\not\equiv0$. % Cf.  \cite[Remark 2.9 ] {cmf2024}.
\end{rem}

The following lemma is concerned with the Fermat type functional equation due to Yang   et al. \cite{ccy2004}.
\begin{lem}\cite[Theorem 2] {ccy2004}
\label{ccy2004}
Let $b$ be a nonzero constant, and $a(z)$ be a meromorphic function. If $a(z)$ is not constant, then
\begin{align*}
f^2(z)+b(f')^2=a(z)
\end{align*}
has no transcendental meromorphic solution $f(z)$ such that $T(r,a)=S(r,f)$.
\end{lem}

The following lemma can be proved by using the similar way as in the proof of \cite[Lemma 2.11]{cmf2024}, we omit the details of the proof.
\begin{lem}
\label{c2.11}
Let $n\ge3$ be an integer, $P_d(z,f)\not\equiv 0$ be a difference polynomial in $f$ of degree $d=n-1$ with small functions of $f$ as its coefficients, and let $p_j,~a_j,~j=1,2,3,$ be nonzero distinct constants.
If difference equation
\begin{align*}
f^n(z)+P_d(z,f)=p_1e^{a_1z}+p_2e^{a_2z}+p_3e^{a_3z}
\end{align*}
admits a finite order meromorphic solution $f$ such that $P_d(z,0)\not\equiv0$, then $m(r,\frac{1}{f})=S(r,f)$, $m(r,\frac{a_j}{f^n})=S(r,f),~~j=1,2,3$.
\end{lem}

\begin{lem}\cite[Lemma 2.3]{zyy2021}
\label{lm3.5}
Let $\lambda$ be a nonzero constant and $k(z)$ be a small entire function of $e^z$. Suppose that $f$ is a solution
of the first-order differential equation
\begin{align*}
    f'-\lambda f=k.
\end{align*}
Let $S_{j }=\{re^{i\theta}:0\le r<\infty,~ \theta_j <\theta<\theta_{j+1} \},~j=1,~2$, where $\theta_1<\pi$,  $\theta_2=\theta_1+\pi$, and $\theta_3=\theta_1+2\pi$. For any given $\epsilon>0$, let $S_{j,\epsilon}=\{re^{i\theta}:0\le r<\infty,~\theta_j+\epsilon<\theta<\theta_{j+1}-\epsilon\},~j=1,~2,$ and $\overline S_{1,\epsilon}$, $\overline S_{2}$ denote the closures of $S_{1,\epsilon},~S_{2}$, respectively.
Then there is an entire function $H$ such that $f=ce^{\lambda z}+H(z)$, where c is constant, and $H$ satisfies
$|H(z)|\le e^{o(1)|z|}$ uniformly in $\overline S_{1,\epsilon}\cup \overline S_{2}$ as $z\to\infty$.
\end{lem}

The next lemma is Borel type theorem, which can be found in \cite{ccy2003}.
\begin{lem}\cite[Theorem 1.51]{ccy2003}
\label{borel}
Let $f_1,\,f_2,\ldots,f_n$ be meromorphic functions, $g_1,\,g_2,\ldots,g_n$ be entire functions satisfying
the following conditions,
\begin{itemize}
    \item [\rm{(i)}]$\sum\limits_{j=1}^n{{f_j(z)}}e^{g_j(z)}\equiv0$,
    \item[\rm{(ii)}]For $1\le{j}<{k}\le{n}$, $g_j-g_k$ is not constant,
    \item[\rm{(iii)}]For $1\le{j}\le{n},1\le{t}<{k}\le{n}$, $T(r,f_j)=o\,\{T(r,e^{g_t-g_k})\},$\, $r\rightarrow\infty,\,r\notin{E}$, where E is the set of finite linear measure.
\end{itemize}
Then $f_j(z)\equiv0,\, j=1,\ldots,n$.
\end{lem}

\begin{lem}\cite{halburd2014}
\label{lm2.7}
Let $f$ be a nonconstant meromorphic function of hyper-order $\rho_2(f)<1$, $k$ be a positive integer and $c$, $h$ be  nonzero complex numbers. Then the following statements hold.
\begin{itemize}
    \item[\rm{(i)}]$m\left(r,\frac{f^{(k)}(z+c)}{f(z)}\right)=S(r,f)$.
    \item[\rm{(ii)}]$N\left(r,\frac{1}{f(z+c)})\right)=N\left(r,\frac{1}{f(z)}\right)+S(r,f)$, $N(r,{f(z+c)})=N(r,{f(z)})+S(r,f).$
\item[\rm{(iii)}]$T(r,f(z+c))=T(r,f(z))+S(r,f)$.
\end{itemize}
\end{lem}

The next result is concerned asymptotic estimation of exponential polynomials.
\begin{lem}\cite[Lemma 2.5.]{mzq2022}
\label{lm2.8}
Let $m$, $q$ be positive integers, $\omega_1,\ldots,\omega_m$ be distinct nonzero complex numbers, and $H_0,\,H_1,$ $\ldots,H_m$ be meromorphic functions of order less than $q$ such that $H_j\not\equiv 0, \, 1\le{j}\le{m}$. Set $\varphi(z)=H_0+\sum\limits_{j=1}^n{{H_j(z)}}e^{\omega_j{z^q}}$. Then the following statements hold.
\begin{itemize}
    \item[\rm{(i)}]There exist two positive numbers $d_1<d_2$, such that for sufficiently large $r$,   $d_1r^q{\le}T(r,\varphi){\le}d_2r^q$.
    \item[\rm{(ii)}]If $H_0\not\equiv0$, then $m\left(r,\frac{1}{\varphi}\right)=o\,(r^q)$.
\end{itemize}
\end{lem}

The next lemma is a differential-difference version of the Clunie lemma.
\begin{lem}[\cite{hpc2019}]
\label{Clunie}
Let $f$ be a transcendental meromorphic solution  of of hyper-order $\rho_2(f)<1$ of the following equation
\begin{align*}
f^nP(z,f)=Q(z,f),
\end{align*}
where $P(z,f)$, $Q(z,f)$ are   differential-difference polynomials with meromorhphic coefficients $\alpha_\lambda\,(\lambda{\in}I)$, such that $m(r,\alpha_\lambda)=S(r,f)$,
$n$ is a positive integer, such that the total degree of $Q(z,f)$  is at most $n$. Then
\begin{align*}
m(r,P(z,f))= S(r,f).
\end{align*}
\end{lem}

The following  inequality shows a relationship between $T(r,f)$ and  $n(r,f)$, which can be found in \cite{3}.
\begin{lem}\cite[P. 97]{3}
\label{lm3.1}
Let $f$ be a meromorphic function. Then
\begin{align*}
    n(t,f)\le\frac{r}{r-t}N(r,f)+O(1)\le\frac{r}{r-t}T(r,f)+O(1),
\end{align*}
where $r>t$ are positive real numbers.
\end{lem}

By using the idea of \cite[Lemma 2.3]{lhf2021}, we get the following result, in which the restricted condition $\rho(f)<\infty$ in \cite[Lemma 2.3]{lhf2021} is relaxed to $\rho_2(f)<1$.

\begin{lem}
 \label{lm2.1}
Let $f$ be a transcendental meromorphic solution of $\rho_2(f)<1$ of the difference equation
\begin{align*}
    f^n(z)+P_d(z,f)=H(z),
\end{align*}
where $n\ge2$ is an integer, $P_d(z,f)$ is a difference polynomial in $f$ of degree
$d \le n-1$ with small functions of $f$ as its coefficients, and $H$ is a meromorphic function satisfying $N(r,H)=S(r,f)$. Then $N(r,f) = S(r,f) $ and $nT(r,f)=T(r,H)+S(r,f)$. Moreover, if the coefficients of $P_d(z,f)$ have no poles and $H(z)$ is an entire function, then $f$ is entire and $nT(r,f)=T(r,H)+S(r,f)$.
\end{lem}

\begin{proof}[Proof] Suppose that the coefficients of $P_d(z,f)$ are  small meromorphic functions of $f$  and $N(r,H)=S(r,f)$.
 Assume that $N(r,f)\not=S(r,f)$ for a contradiction. Then, let $z_0$ be a pole of $f$, which is not the pole of the coefficients of $P_d(z,f)$ and $H(z)$.
Let $P_d(z,f)=\sum\limits_{\lambda\in I}b_\lambda(z)[f(z+c_{1})]^{\lambda_1}\cdots[f(z+c_s)]^{\lambda_s}$, where $c_j, ~j = 1,\dots,s$, are distinct nonzero constants, $ \lambda_j,  ~j = 1,\dots,s$, are
nonnegative integers, $ \lambda= (\lambda_0,\dots,\lambda_s)$, $I$ is a finite index set of $\lambda$, $b_\lambda$
are small meromorphic functions of $f$.

For any given $\epsilon\in(0, \frac{1}{2})$, by using  similar way as in the proof of \cite[Lemma 2.3]{lhf2021}, we can obtaine a sequence $\{z_l\}$ of poles of $f$ with $f(z_l)=\infty^{k_l}$ and $k_l>\left(\frac{n-\epsilon}{n-1}\right)^{l-1}$, where $l\ge1$ is a positive integer, $z_l=z_0+c_{s_1}+\cdots+c_{s_l}$, $c_{s_i}\in\{c_1,\dots,c_s\}, ~i=1,\dots,l.$
Let $B(0, r)$ denote a disc of radius $r$ around $0$, $c=\max\limits_{1\le j\le s}\{|c_j|\}$. Since $z_l\in B(0,|z_0|+(l-1)c)$, set $r_l=|z_0|+(l-1)c$, then $n(r_l,f)>k_l>\left(\frac{n-\epsilon}{n-1}\right)^{l-1}.$

Let $\rho_0=\rho_2(f)<1$, thus $\log T(r,f)=O(r^{\rho_0})$. Let $r=2r_l$, by Lemma \ref{lm3.1}, we get
\begin{align*}
  0=\overline{\lim\limits_{r\to\infty}}{\frac{\log^+T(r,f)}{r}}
  \ge \overline{\lim\limits_{l\to\infty}}{\frac{\log^+\frac{1}{2}n(r_l,f) }{2r_l}}
  \ge\overline{\lim\limits_{l\to\infty}}{\frac{(l-1)\log^+\left(\frac{n-\epsilon}{n-1}\right) }{2|z_0|+2(l-1)c}} >0,
\end{align*}
which is impossible. Thus $N(r,f)=S(r,f)$.
By Lemma \ref{lm2.7} and the proof of \cite[Theorem 1.12]{ccy2003}, it is easy to have $nm(r,f)=m(r,f^n)=m(r,f^n+P_d(z,f))=m(r,H)+S(r,f)$. So  $nT(r,f)=T(r,H)+S(r,f)$.

If the coefficients of $P_d(z,f)$ have no poles, $H(z)$ is entire, we will prove $f$ is entire. Otherwise, suppose $z_0$ is a pole of $f$ for a contradiction. Then it is easy to get a sequence $\{z_l\}$ of poles of $f$ with $f(z_l)=\infty^{k_l}$, where $z_l=z_0+c_{s_1}+\cdots+c_{s_l}$, $c_{s_i}\in\{c_1,\dots,c_s\}, ~i=1,\dots,l$, $k_l\ge \left(\frac{n}{n-1}\right)^{(l-1)}$, $l\ge1$ is a positive integer. Then by using the same method above, we also get a contradiction. Then $f$  is an entire function. By Lemma \ref{lm2.7}, it is easy to have $nm(r,f)=m(r,H)+S(r,f)$. So  $nT(r,f)=T(r,H)+S(r,f)$. This proof is completed.
\end{proof}

\section{Proof  of Theorem  1.1}
\begin{proof}
Suppose that $f$ is a meromorphic solution of Eq. \eqref{F} with $\rho_2(f)<1$. Let $P_d(z,f)=P$. Then we rewrite \eqref{F} as
\begin{align}
\label{3.1}
f^n+P=p_1e^{a_1z}+p_2e^{a_2z}+p_3e^{a_3z}.
\end{align}
By the Lemma \ref{lm2.1}, Lemma \ref{lm2.7} and Lemma \ref{lm2.8}, we have $N(r,f)=S(r,f)$, $m(r,f)=O(r)$. So $T(r,f)=O(r)$.

Differentiating \eqref{3.1}: $nf^{n-1}f'+P'=p_1a_1e^{a_1z}+p_2a_2e^{a_2z}+p_3a_3e^{a_3z}$.  then eliminating $e^{a_2z}$ yields
\begin{align}
\label{3.2}
a_2f^n-nf^{n-1}f'+a_2P-P'=(a_2-a_1)p_1e^{a_1z}+(a_2-a_3)p_3e^{a_3z}.
\end{align}

Differentiating \eqref{3.2}: $na_2f^{n-1}f'-n(n-1)f^{n-2}(f')^2-nf^{n-1}f''+a_2P'-P''=a_1(a_2-a_1)p_1e^{a_1z}+a_3(a_2-a_3)p_3e^{a_3z}$. Then eliminating $e^{a_3z}$ or $ e^{a_1z}$ yields
\begin{align}
\label{3.3}
a_1a_2f^n-n(a_1+a_2)f^{n-1}f'+n(n-1)f^{n-2}(f')^2+nf^{n-1}f''+H_1
=(a_2-a_3)(a_1-a_3)p_3e^{a_3z},
\end{align}
where $H_1=a_1a_2P-(a_1+a_2)P'+P''$.
\begin{align}
\label{3.4}
a_2a_3f^n-n(a_2+a_3)f^{n-1}f'+n(n-1)f^{n-2}(f')^2+nf^{n-1}f''+H_2
=(a_3-a_1)(a_2-a_1)p_1e^{a_1z},
\end{align}
where $H_2=a_2a_3P-(a_2+a_3)P'+P''$.

Using the same idea, we have
\begin{align}
\label{3.5}
a_1a_3f^n-n(a_1+a_3)f^{n-1}f'+n(n-1)f^{n-2}(f')^2+nf^{n-1}f''+H_3
=(a_3-a_2)(a_1-a_2)p_2e^{a_2z},
\end{align}
where $H_3=a_1a_3P-(a_1+a_3)P'+P''$.

Rewriting \eqref{3.5} as
\begin{align}
\label{3.6}
f^{n-2}G(z)+H_3=(a_3-a_2)(a_1-a_2)p_2e^{a_2z},
\end{align}
where
\begin{align}
\label{3.7}
G=a_1a_3f^2-n(a_1+a_3)ff'+n(n-1)(f')^2+nff''.
\end{align}

Below, we discuss $G\equiv0$ or $G\not\equiv0$.

\setcounter{case}{0}
\begin{case}
\rm{
If $G\equiv0$.  By the \eqref{3.6}, we have
\begin{align}
\label{3.8}
H_3=a_1a_3P-(a_1+a_3)P'+P''=(a_3-a_2)(a_1-a_2)p_2e^{a_2}.
\end{align}

Solving the \eqref{3.8}, we have
\begin{align}
\label{3.9}
P_d(z,f)=P=p_2e^{a_2z}+c_1e^{a_1z}+c_3e^{a_3z},
\end{align}
where $c_3,~c_1$ are constants. Substituting \eqref{3.9} into \eqref{F}, we obtain
\begin{align}
\label{3.10}
f^n(z)=(p_1-c_1)e^{a_1z}+(p_3-c_3)e^{a_3z}.
\end{align}

From Theorem E and \eqref{3.10}, $p_3-c_3=0$ or $p_1-c_1=0$, that is $f(z)=\gamma_1e^{\frac{a_1z}{n}}$, $c_3=p_3$  $\gamma_1^n=p_1-c_1$ or $f(z)=\gamma_3e^{\frac{a_3z}{n}}$, $c_1=p_1$  $\gamma_3^n=p_3-c_3$ .

If $f(z)=\gamma_1e^{\frac{a_1z}{n}}$, substituting $f$ into \eqref{3.9}, and combining Lemma \ref{borel}, we get
\[c_1=0,~\frac{a_1}{a_2}=\frac{n}{d_1},~\frac{a_1}{a_3}=\frac{n}{d_2},~\gamma_1^n=p_1,~P_d(z,f)=p_2e^{a_2z}+p_3e^{a_3z}. \]
If $f(z)=\gamma_3e^{\frac{a_3z}{n}}$, substituting $f$ into \eqref{3.9}, and combining Lemma \ref{borel}, we get
\[c_3=0,~\frac{a_3}{a_2}=\frac{n}{d_3},~\frac{a_3}{a_1}=\frac{n}{d_4},~\gamma_3^n=p_3,~P_d(z,f)=p_2e^{a_2z}+p_1e^{a_1z}, \] where $~d_i\in\{1,\dots,d\}$ are integers$~i=1,2,3,4$.
}
\end{case}

\begin{case}
\rm{
If $G\not\equiv0$. we will prove $T(r,G)=S(r,f)$  first.  Without loss of generality, we assume that the  $P_d(z,0)\not=0$ is not equiv $0$.  Otherwise, we make
the transformation $g+a=f$ for a suitable constant $a$, such $a^n+P_d(z,a)\not=0$.

From the assumptions of the theorem, $0<\frac{a_2}{a_1}=t\le\frac{n-2}{n}<1$, where $t$ is a positive constant. From the Lemma \ref{c2.11} , we know $m(r,\frac{1}{f})=S(r,f)$, $m(r,\frac{e^{a_iz}}{f^n})=S(r,f),~i=1,2,3$.

By the \eqref{3.6}, we get
\begin{align}
\label{3.11}
    G(z)&=(a_3-a_2)(a_1-a_2)p_2\frac{e^{a_2z}}{f^{n-2}}-\frac{H_3}{f^{n-2}}
   % &=(a_3-a_2)(a_1-a_2)p_2\left(\frac{a_1z}{f^n}\right)^{\frac{t}{s}}\left(\frac{1}{f^{(n-2)s-nt}}\right)^{\frac{1}{s}}-\frac{H_3}{f^{n-2}}.
\end{align}

Let $z=re^{i\theta},~r>0,~\theta\in[0,2\pi)$, and $E_1=\{\theta\in[0,2\pi):|e^{a_1z}|\le1\}$, $E_2=\{\theta\in[0,2\pi):|e^{a_1z}|>1\}$. If $\theta\in E_1$, then $|e^{a_2z}|=|e^{a_1z}|^t\le1$, thus $\frac{|e^{a_2z}|}{|f^{n-2}|}\le\frac{1}{|f^{n-2}|}$. If $\theta\in E_2$, $\frac{|e^{a_2z}|}{|f^{n-2}|}=\frac{|e^{a_1z}|^t}{|f^{n-2}|}\le\frac{|e^{a_1z}|^{\frac{n-2}{n}}}{|f^{n-2}|}=\left(\frac{|e^{a_1z}|}{|f^{n}|}\right)^\frac{n-2}{n}$. Combining $m(r,\frac{1}{f})=S(r,f)$, $m(r,\frac{e^{a_1z}}{f^n})=S(r,f)$, we have $m(r,\frac{e^{a_2z}}{f^{n-2}})=S(r,f)$. By Lemma \ref{lm2.7} and $m(r,\frac{1}{f})=S(r,f)$, it is easy to see $m(r,\frac{H_3}{f^{n-2}})=S(r,f)$. Thus $m(r,G)=S(r,f)$, with $N(r,G)=S(r,f)$, we have proved $T(r,G)=S(r,f)$.

% Combining   $s(n-2)\ge nt$, Lemma \ref{lm2.7} and \eqref{3.11}, we have $m(r,G)=S(r,f)$. Now it is easy to see $T(r,G)=S(r,f)$.

Differentiating $G(z)$ yields $G'(z)=2a_1a_3ff'-n(a_1+a_3)ff''-n(a_1+a_3)(f')^2+n(2n-1)f'f''+nff'''$. Multiply $G$ and $G'$:
\begin{align}
\label{3.12}
(a_1a_3G')f^2-[n(a_1+a_3)G'+2a_1a_3G]ff'+[nG'+n(a_1+a_3)G]ff''-(nG)ff'''\\ \nonumber
=[n(2n-1)G]f'f''-[n(n-1)G'+n(a_1+a_3)G](f')^2.
\end{align}

By \eqref{3.7}, if $z_0$ is a multiple zero of $f$, then $z_0$ must be a zero of $G$. Thus $N_1(r,\frac{1}{f})=T(r,f)+S(r,f)$, where $N_1(r,\frac{1}{f})$ denotes the counting
function corresponding to simple zeros of $f$.
Suppose $z_1 $  is a simple zero of $f$  and not the pole of the  coefficients of \eqref{3.12}, by \eqref{3.12} $z_1$ must be a zero of $(2n-1)Gf''-[(n-1)G'+(a_1+a_3)G]f'$. Let
\begin{align}
\label{3.13}
A=\frac{(2n-1)Gf''-[(n-1)G'+(a_1+a_3)G]f'}{f}.
\end{align}

Then we have $T(r,A)=S(r,f)$. Subsisting \eqref{3.13} into \eqref{3.7} and eliminating $f''$ yields
\begin{align}
\label{3.14}
q_1f^2+q_2ff'+q_3(f')^2=G,
\end{align}
where $q_1=a_1a_3+\frac{nA}{(2n-1)G}$, $q_2=\frac{n(n-1)}{2n-1}[\frac{G'}{G}-2(a_1+a_3)]$, $q_3=n(n-1)$. By Lemma \ref{lm2.2} and \eqref{3.14}, we have
\begin{align}
\label{3.15}
    q_3(q_2^2-4q_1q_3)\frac{G'}{G}=q_3(q_2^2-4q_1q_3)'-q_2(q_2^2-4q_1q_3).
\end{align}
}
\end{case}

Next we will discuss whether $q^2\equiv4q_1q_3$.

\setcounter{subsection}{2}
\setcounter{subcase}{0}
\renewcommand{\thesubcase}{\arabic{subsection}.\arabic{subcase}}
\begin{subcase}
\rm{Suppose $q_2^2\equiv4q_1q_3$. We claim $q_2\not\equiv0$, otherwise $q_1\equiv0$. By \eqref{3.14}, which implies $q_3(f')^2=G$,  which is a contradiction.  Thus use the similar methods in the prove of \cite[Subcase 2.2 of Theorem 1.1]{cmf2024}, we can also get  $f''=\left(\frac{1}{2}\frac{G'}{G}-\frac{q_2}{2q_3}\right)f'-\frac{1}{q_2}\left(q_1'-q_1\frac{G'}{G}\right)f$ (\cite[equation (3.24)]{cmf2024} ), we omit the details. Substituting $q_2^2=4q_1q_3$ into the above equation of $f''$, then
\begin{align}
\label{3.16}
f''=\frac{1}{2n-1}\left[(n-1)\frac{G'}{G}+(a_1+a_3)\right]f'-\frac{1}{2(2n-1)}\left[\frac{G'}{G}-\frac{1}{2}\left(\frac{G'}{G}\right)^2+(a_1+a_3)\frac{G'}{G}\right]f.
\end{align}
By \eqref{3.16} and \eqref{3.13} that
\begin{align}
\label{3.17}
\frac{A}{G}=-\frac{1}{2}\left[\left(\frac{G'}{G}\right)'-\frac{1}{2}\left(\frac{G'}{G}\right)^2+(a_1+a_3)\frac{G'}{G}\right].
\end{align}

Now we claim $G'\equiv0$. Otherwise,  suppose $G'\not\equiv0$, differentiating \eqref{3.17}, we get
\begin{align}
\label{3.18}
  \left(\frac{A}{G}\right)'=-\frac{1}{2}\left[\left(\frac{G'}{G}\right)''-\frac{1}{2}\left(\frac{G'}{G}\right)\left(\frac{G'}{G}\right)'+(a_1+a_3)\left(\frac{G'}{G}\right)'\right].
\end{align}
Differentiating $q^2\equiv4q_1q_3$, get $q_2q_2'=2q_1'q_3$, that is
\begin{align}
\label{3.19}
\left(\frac{A}{G}\right)'=\frac{n-1}{2(2n-1)}\left(\frac{G'}{G}\right)'\left[\frac{G'}{G}-2(a_1+a_3)\right].
\end{align}
Denote $\phi=\frac{G'}{G}$. By \eqref{3.18} and \eqref{3.19}, we have
\begin{align}
\label{3.20}
(a_1+a_3)\phi'=n\phi\phi'-(2n-1)\phi''.
\end{align}

If $\phi'\equiv0$, then $G=c_5e^{c_6z}$, $c_5,c_6$ are constants. Because $G'\not\equiv0$, so $c_5c_6\not=0$. Which means $T(r,G)=O(r)$ a contradiction.

If $\phi'\not\equiv0$, from \eqref{3.20} that
\begin{align*}
e^{(a_1+a_3)z}=c_7G^n(\phi')^{-(2n-1)},~~c_7\in\mathbb{C}\setminus\{0\}
\end{align*}
}
which implies $a_1+a_3=0$. Subsisting $a_1+a_3=0$ into \eqref{3.20}, then
\begin{align}
\label{3.21}
 c_7G^n=\left[\left(\frac{G'}{G}\right)'\right]^{2n-1}.
\end{align}
If $z_3$ is a  zero of $G$, then $z_3$ is a pole of $\frac{G'}{G}$. Subsisting $z_3$ into \eqref{3.21}, we get a contradiction. If $z_4$ is a pole of $G$ of multiplicity $k$, where $k$ is a positive integer, then $z_4$ is a pole of $\frac{G'}{G}$ of multiplicity $1$. Subsisting $z_4$ into \eqref{3.21}, we get $kn=4n-2$. Which implies $n=\frac{-2}{k-4}$, so $0<k\le3$. We get $n=2$ or $n=1$, this contradicts the hypothesis $n\ge 3$. Thus $G$ does not have  zero and pole. Since $T(r,G)=o(r)$, thus $G$ is a constant, which  contradicts with $G'\not\equiv0$.

Now we get $G'\equiv0$, then by \eqref{3.17} $A\equiv0$.  By \eqref{3.13}, then $(2n-1)\frac{f''}{f'}=a_1+a_3$. Which implies $f(z)=c_8e^{\frac{(a_1+a_3)z}{2n-1}}+c_9,$ where $~c_8, ~c_9\in\mathbb{C}\setminus 0$. Subsisting $f$ into \eqref{F}, by the Lemma 2.6, we have the follow asserts.

If $\frac{n(a_1+a_3)}{2n-1}=a_1$, we get $\frac{a_1}{a_3}=\frac{n}{n-1}$.  Thus $f(z)=t_1e^{\frac{a_1z}{n}}+t_0$, where $t_0\not=0$ is a constant, $t_1^n=p_1,$ $-t_0^n$ equals to  the constant term of $P_d(z,t_1e^{\frac{a_1z}{n}}+t_0)$.

If $\frac{n(a_1+a_3)}{2n-1}=a_2$, we get $\frac{d_5(a_1+a_3)z}{2n-1}=a_1$, where $0<d_5\le n-2$ is a integer. Which implies $d_5a_2=na_1$, this contradicts the hypothesis $\frac{a_2}{a_1}\le\frac{n-2}{n}$.

If $\frac{n(a_1+a_3)}{2n-1}=a_3$, we get $\frac{a_1}{a_3}=\frac{n-1}{n}$.  Thus $f(z)=t_2e^{\frac{a_3z}{n}}+t_3$, where $t_3\not=0$ is a constant, $t_2^n=p_3,$ $-t_3^n$ equals the constant term of $P_d(z,t_2e^{\frac{a_3z}{n}}+t_3)$.

\end{subcase}

\begin{subcase}
\rm{
Suppose $q_2^2\not\equiv4q_1q_3$. From \eqref{3.15} that
\begin{align*}
 2(a_1+a_3)=2n\frac{G'}{G}-(2n-1)\frac{(q^2-4q_1q_3)'}{q^2-q_1q_3}
\end{align*}
then there exists $c_3\not=0$ such that $e^{2(a_1+a_3)z}=c_3G^{2n}(q^2-4q_1q_3)^{-(2n-1)}.$ Thus $a_1=-a_3$, then  $f^{n-2}G-a_1a_3f^n=n(n-1)f^{n-2}(f')^2+nf^{n-1}f''$. Substituting $f^{n-2}G-a_1a_3f^n$ into \eqref{3.3} and \eqref{3.5}, we have
\begin{align*}
a_1f^n-nf^{n-1}f'+\frac{f^{n-2}G}{a_1+a_2}+\frac{H_1}{a_1+a_2}=2a_1p_3e^{a_3z}
\end{align*}
\begin{align*}
a_1f^n+nf^{n-1}f'+\frac{f^{n-2}G}{a_1-a_2}+\frac{H_2}{a_1-a_2}=2a_1p_1e^{a_1z}
\end{align*}

Multiplying $2a_1p_3e^{a_3z}$ and $2a_1p_1e^{a_1z}$, we have
\begin{align}
\label{3.22}
a_1^2f^{2n}-n^{2}f^{2(n-1)}(f')^2=4a_1^2p_1p_3+F(z,f),
\end{align}
where $F(z,f)$ is a differential-difference polynomial of $f$ of degree at most $2n-2$. Further write the equation \eqref{3.22} in
to $f^{2n-2}\Omega(z)=4a_1^2p_1p_3+F(z,f)$, where $\Omega(z)=a_1^2f^2-n^2(f')^2$. By the Lemma \ref{Clunie}, we have $m(r,\Omega)=S(r,f)$. Combining $N(r,f)=S(r,f)$, yields $T(r,\Omega)=S(r,f)$. By Lemma \ref{ccy2004}, then $\Omega$ is a constant. So differentiating $\Omega$, we have
\begin{align}
\label{3.23}
2a_1^2f=2n^2f''.
\end{align}

Solving the \eqref{3.23}, we have
\begin{align*}
    f(z)=\gamma_4e^{\frac{a_1z}{n}}+\gamma_5e^{\frac{a_3z}{n}},
\end{align*}
where $\gamma_4$ and $\gamma_5$ are nonzero constant. Subsisting $f$ in \eqref{F}, then $\gamma_4^n=p_1$, $\gamma_5^n=p_3$.

It's worth noting that we started Case 2 with the assumption $P_d(z,0)\not\equiv0$. So if $P_d(z,0)=0$, we have
\begin{align*}
   f(z)=\gamma_4e^{\frac{a_1z}{n}}+\gamma_5e^{\frac{a_3z}{n}}+\gamma_6,
\end{align*}
where $\gamma_6$ is a constant.
}
\end{subcase}
\end{proof}

\section{Proofs of Theorems 1.2 and 1.3}
\begin{proof}[Proof of Theorem 1.2] Suppose that equation \eqref{F} admits a meromorphic solution $f$ with $\rho_2(f)<1$.
Let $P_d(z,f)=P$, then we rewrite \eqref{F} as
\begin{align}
\label{4.1}
f^n+P=p_1e^{a_1z}+p_2e^{a_2z}+p_3e^{a_3z}.
\end{align}
By Lemma \ref{lm2.7} and Lemma \ref{lm2.8}, we have $m(r,f)=O(r)$. So $f$ is transcendental. By Lemma \ref{lm2.1}, $f$ is an entire function and $T(r,f)=O(r)$.

 Without loss of generality, we assume of $P_d(z,0)\not\equiv0$.  Otherwise, we make
the transformation $g+a=f$ for a suitable constant $a$, such $a^n+P_d(z,a)\not=0$. Then \eqref{4.1} is converted
to the form $g^n+Q(z,g)= p_1e^{a_1z}+p_2e^{a_2z}+p_3e^{a_3z}$, where $Q(z,g)$ is a difference polynomial in $g$ of degree $d=n-1$ and $Q(z,0)=a^n+P_d(z,a)\not\equiv0$.

 By Lemma \ref{c2.11}, we have $m(r,\frac{1}{f})=S(r,f)$, and $m(r,\frac{a_j}{f^n})=S(r,f),~j=1,2,3$.
Since $\frac{a_1}{a_2}\in \{ \frac{s}{n}:1\le s \le n-1\}$,  $\frac{a_3}{a_2}\in \{ \frac{t}{n}:1\le t \le n-1\}$, where $s,t$ are positive integer. $\frac{e^{a_1z}}{f^{n-1}}=\left(\frac{e^{a_2z}}{f^n}\right)^{\frac{s}{n}}\frac{1}{f^{n-1-s}}$. $\frac{e^{a_3z}}{f^{n-1}}=\left(\frac{e^{a_2z}}{f^n}\right)^{\frac{t}{n}}\frac{1}{f^{n-1-t}}$, then we have $m(r,\frac{e^{a_1z}}{f^{n-1}})=S(r,f)$ and $m(r,\frac{e^{a_3z}}{f^{n-1}})=S(r,f)$.

Differentiating \eqref{4.1}  and eliminating $e^{a_2z}$ yields
\begin{align}
\label{4.2}
f^{n-1}(a_2f-nf')+a_2P-P'=(a_2-a_1)p_1e^{a_1z}+(a_2-a_3)p_3e^{a_3z}.
\end{align}

 Let $E_1=\{\theta\in[0,2\pi):|f(re^{i\theta})|\le1\}$, $E_2=\{\theta\in[0,2\pi):|f(re^{i\theta})|>1\}$, and
 \begin{align}
 \label{4.3}
\phi(z)=a_2f-nf'.
 \end{align}

If $\theta_1\in E_1$, then $ \phi(re^{i\theta_1}) = a_2f(re^{i\theta_1})-f(re^{i\theta_1})\frac{f'(re^{i\theta_1})}{f(re^{i\theta_1})}$. By logarithmic derivative lemma and $|f(re^{i\theta})|\le1$, we have  $\int_{E_1}\log^{+}|\phi(re^{i\theta_1})|d\theta=S(r,f)$.
If  $\theta_1\in E_2$, then rewrite \eqref{4.3} into  $\phi=-\frac{a_2P-P'}{f^{n-1}}+(a_2-a_1)p_1\frac{e{^{a_1z}}}{f^{n-1}}+(a_2-a_3)p_3\frac{e^{a_3z}}{f^{n-1}}$. By Lemma \ref{lm2.7}, $|f(re^{i\theta})|>1$, $m(r,\frac{e^{a_1z}}{f^{n-1}})=S(r,f)$ and $m(r,\frac{e^{a_3z}}{f^{n-1}})=S(r,f)$, then $\int_{E_2}\log^{+}|\phi(re^{i\theta_1})|d\theta=S(r,f)$. So $\phi$ is an entire function satisfies $m(r,\phi)=S(r,f)$.

Using Lemma \ref{lm3.5}  to \eqref{4.3},  we get $f(z)=c_1e^{\frac{a_2z}{n}}+\gamma(z)$, where $|\gamma(z)|\le e^{o(1)|z|}$ uniformly in $\overline S_{1,\epsilon}\cup \overline S_{2}$ as $z\to\infty$, $c_1$ is a constant. Because $T(r,f)=O(r)$, so $\rho(f)=\mu(f)=1$, $\rho(\gamma)\le1$. If $\rho(\gamma)<1$
, then $\gamma$ is a small function of $f$. If $\rho(\gamma)=1$, then by applying the Phragm\'en-Lindel\"of theorem (see \cite[Theorem 7.3]{holl}) to the $\gamma(z)$, we have $|\gamma(z)|\le e^{o(1)|z|}$ uniformly as $z\to\infty$ in $\mathbb{C}$, so $\gamma$ is  also a small function of $f$.
Subsisting $f$ into \eqref{F}, we have $c_1^n=p_2$. This completes the proof.
\end{proof}

\begin{proof}[Proof of Theorem 1.3]
Suppose $f$ is a meromorphic solution \eqref{th1.3} such that
$T(r,p_j)$$=S(r,f),$ $~j=1,2,3$. We aim for a contradiction. It is easy to see
$T(r,f)\le O(r)+S(r,f)$  and $N(r,f)=S(r,f)$. Thus the small function of
$f$ is also the small function of $e^{z}$. By Lemma \ref{lm2.8}, we have  $T(r,f)=O(r)$,
write \eqref{th1.3} as
\begin{align*}
(fe^{\frac{-a_3z}{n}})^n-p_3=p_1e^{(a_1-a_3)z}+p_2e^{(a_2-a_3)z}.
\end{align*}

Let $g=fe^{\frac{-a_3z}{n}}$, $b_1=a_1-a_3\not=0$, $b_2=a_2-a_3\not=0$. Then
\begin{align}
\label{4.4}
g^n-p_3=p_1e^{b_1z}+p_2e^{b_2z}.
\end{align}
By Lemma \ref{lm2.8}, we have
\begin{align}
\label{4.5}
T(r,g)=O(r),~m(r,\frac{1}{g})=S(r,g).
\end{align}
Thus $S(r,g)=S(r,f)=S(r,e^{z})$,  then $~N(r,g)=S(r,g).$
Differentiating \eqref{4.4} yields
\begin{align}
\label{4.6}
ng^{n-1}g'-p_3'=(p_1'+b_1p_1)e^{b_1z}+(p_2'+p_2b_2)e^{b_2z}.
\end{align}
Eliminating $e^{b_2z}$ from \eqref{4.6} and \eqref{4.4}, we have
\begin{align}
\label{4.7}
    np_2g^{n-1}g'-(p_2'+p_2b_2)g^n+P=Ae^{b_1z},
\end{align}
where $P=-p_2p_3'+p_3(p_2'+b_2p_2)$, $A=p_2(p_1'+b_1p_1)-p_1(p_2'+b_2p_2)$.
We claim $A\not\equiv0$, $P\not\equiv0$. Otherwise, if $A\equiv0$, then $b_1-b_2=\frac{p_2'}{p_2}-\frac{p_1'}{p_1}$. By integration, we get $e^{(b_1-b_2)z}=c\frac{p_2}{p_1}, c \in\mathbb{C}\setminus0$, which implies $b_1=b_2$, it is a contradiction. Thus $A\not\equiv0$.
Using the same idea we have $P\not\equiv0$.
Differentiating \eqref{4.7} yields
\begin{align}
\label{4.8}
np_2g^{n-2}(g')^2+np_2'g^{n-1}g'&+np_2g^{n-1}g''-(p'+p_2b_2)'g^n+(p'+p_2b_2)g^{n-1}g'+P'\\ \nonumber
&=(A'+Ab_1)e^{b_1z}.
\end{align}
Eliminating $e^{b_1z}$ from \eqref{4.7} and \eqref{4.8}, then
\begin{align}
\label{4.9}
    g^{n-2}(d_1g^2+d_2gg'+d_3(g')^2+d_4gg'')=R,
\end{align}
where $d_1=(p'+p_2b_2)(A'+Ab_1)-A(A'+Ab_1)'$, $d_2=-np_2(b_1+b_2)A-np_2A'$, $d_3=np_2A$, $d_4=np_2A$, $R=(A'+Ab_1)P-AP'$.

Let $Q(z,g)=d_1g^2+d_2gg'+d_3(g')^2+d_4gg''$, we claim $Q(z,g)\not\equiv0$. Otherwise $R\equiv0=(A'+Ab_1)P-AP'$, by integration $\frac{A}{P}=c_2e^{b_1z},~ c_2 \in\mathbb{C}\setminus0$, it is a contradiction.

By \eqref{4.9} and $m(r,\frac{1}{g})=S(r,g)$, we get $m(r,Q)=S(r,g)$. Thus $T(r,Q)=S(r,g)$.  $T(r,g^{n-2})=T(r,\frac{R}{Q})=S(r,g)$, note that $n\ge3$, we get a contradiction, and then the proof is completed.
\end{proof}

\section*{Declarations}
\begin{itemize}
\item \noindent{\bf Funding}
This research work is supported by the National Natural Science Foundation of China (Grant No. 12261023, 11861023), the Foundation of Science and Technology project of Guizhou Province of China (Grant No. [2018]5769-05).

\item \noindent{\bf Conflicts of Interest}
The authors declare that there are no conflicts of interest regarding the publication of this paper.

\end{itemize}

%\noindent{\bf Data Availability}

%\noindent The data used to support the findings of this study are included within the article.

%\noindent{\bf Conflicts of Interest}

%\noindent The authors declare that there are no conflicts of interest regarding the publication of this paper.

%\noindent{\bf Authors' Affiliations}

%\noindent The all authors come from School of Mathematical Sciences, Guizhou Normal University, China.

%\noindent{\bf Authors' Contributions}

%\noindent The all authors have the equal contributions for the manuscript. All authors read and approved the final manuscript.

\end{document}